\documentclass[a4paper,10pt]{amsart}

\usepackage{amsmath,amssymb,amsthm,a4wide}

\usepackage{tikz}
\usepackage{tkz-euclide}
\usetikzlibrary{intersections}

\usepackage{graphicx}
\usepackage{graphics}

\newtheorem*{thm}{Theorem}
\newtheorem{lemma}{Lemma}
\newtheorem*{corollary}{Corollary}
\newtheorem*{proposition}{Proposition}

\begin{document}

\title[well-distributed great circles on {\small $\mathbb{S}^2$}]{ Well-distributed great circles on {\Large $\mathbb{S}^2$}}

\author[]{Stefan Steinerberger}
\address[Stefan Steinerberger]{Department of Mathematics, Yale University, 06511 New Haven, CT, USA}
\email{stefan.steinerberger@yale.edu}

\begin{abstract} Let $C_1, \dots, C_n$ denote the $1/n-$neighborhood of $n$ great circles on $\mathbb{S}^2$.
We are interested in how much these areas have to overlap and prove the sharp bounds
$$  \sum_{i, j = 1 \atop i \neq j}^{n}{|C_i \cap C_j|^s} \gtrsim_s  \begin{cases} n^{2 - 2s} \qquad &\mbox{if}~0 \leq s < 2 \\ n^{-2} \log{n} \qquad &\mbox{if}~s = 2\\ n^{1- 3s/2} \qquad &\mbox{if}~s > 2. \end{cases} .$$
For $s=1$ there are arrangements for which the sum of mutual overlap is uniformly bounded (for the analogous problem in $\mathbb{R}^2$ the lower bound is $\gtrsim \log{n}$) and there are strong connections to minimal energy configurations of $n$ charged electrons on $\mathbb{S}^2$ (the J. J. Thomson problem). 
\end{abstract}


\maketitle
\vspace{-20pt}

\section{Introduction and statement of results}
An old problem of Fejes-T\'oth \cite{fejes} is how to arrange $n$ great circles on $\mathbb{S}^2$ so as to minimize the largest distance between
any point on $\mathbb{S}^2$ and its nearest great circle. He conjectures the best arrangement to be one where all $n$ great circles pass through
the same point (the 'north pole') in an equal-angle arrangement. This has only been proven for $n=3$ \cite{rosta} and $n=4$ \cite{vier}.
The problem is motivated by the distribution of satellites in orbit (\cite{fejes} is titled 'Exploring a planet') and has a natural reformulation: assume
$\delta-$neighborhoods of $n$ great circles cover $\mathbb{S}^2$, how small can $\delta$ be? The Fejes-T\'oth's conjecture then states $\delta \geq \pi/(2n)$ (see
\cite{fodor} for known results).

\begin{center}
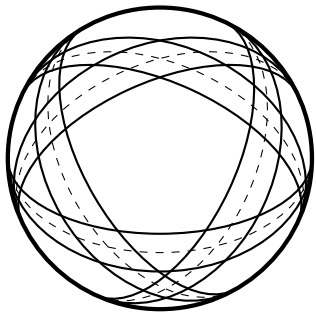
\begin{figure}[h!]
\begin{tikzpicture}
\draw[ultra thick] (0,0) circle (2) ;
\foreach \angle[count=\n from 1] in {0,65,155,205,300} {
    \begin{scope}[rotate=\angle]
    \path[thick, draw] (-2,0) arc [start angle=180,
                        end angle=360,
                        x radius=2cm,
                        y radius=1cm] ;
                            \path[dashed,draw] (-2,0) arc [start angle=180,
                        end angle=360,
                        x radius=2cm,
                        y radius=1.25cm] ;
                            \path[thick,draw] (-2,0) arc [start angle=180,
                        end angle=360,
                        x radius=2cm,
                        y radius=1.5cm] ;
    \end{scope}
    };
\end{tikzpicture}
\caption{Great circles, their $1/n-$neighborhoods and intersection pattern.}
\end{figure}
\end{center}
\vspace{-20pt}
We are interested in the extent to which neighborhoods of great circles have to overlap -- or, in the spirit of Fejes-T\'oth, how to position $n$ satellites so as to minimize the chances of collision. We will consider $n$ great circles on $\mathbb{S}^2$ and  
denote their $1/n-$neighborhoods by $C_1, \dots, C_n$ (the scaling $1/n$ is natural and simplifies exposition, however, our methods also apply to other widths; see below). Any two great circles intersect, so it is natural to ask to which extent the neighborhoods can manage to avoid each other. A natural measure for this is
$$ \sum_{i, j = 1 \atop i \neq j}^{n}{|C_i \cap C_j|} \qquad \mbox{which can be seen to be}~\gtrsim 1~\mbox{as follows:}$$
observe that any two great circles meet (creating $\sim n^2$ pairs of intersections) and $|C_i \cap C_j|$ is minimized if their great circles meet transversally which creates an area of $|C_i \cap C_j| \gtrsim n^{-2}$. One natural special case is the arrangement where all great circles meet in a point -- in that case the global geometry of $\mathbb{S}^2$ is of less importance and we may turn to the analogous problem in $\mathbb{R}^2$. 

\begin{proposition} Let $\ell_1, \dots, \ell_n$ be any set of $n$ lines in $\mathbb{R}^2$ such that any two lines intersect in some point and denote their $1/n-$neighborhoods by $T_1, \dots, T_n$. Let $s \geq 0$, then
there exists $c_s > 0$ such that
$$ \sum_{i, j = 1 \atop i \neq j}^{n}{|T_i \cap T_j|^s}   \geq c_s \begin{cases} n^{2 - 2s} \qquad &\mbox{if}~0 \leq s < 1 \\ \log{n} \qquad &\mbox{if}~s = 1\\ n^{1- s} \qquad &\mbox{if}~s > 1. \end{cases}.$$ 
\end{proposition}
Since there are $\sim n^2$ intersection with a contribution of at least $|T_i \cap T_j|^{s} \gtrsim n^{-2s}$, we observe that the trivial lower bound is sharp for $0 \leq s < 1$ and off by a logarithmic factor for $s=1$. 
The result is related in spirit to A. Cordoba's work \cite{cordoba} on the two-dimensional Kakeya problem (which uses that if any two of the lines meet an angle at least $\gtrsim n^{-1}$, then the 
size of the quantity ($s=1$) is $\lesssim \log{n}$).
In particular,  the total overlap when all great circles meet at an equal angle is of order $\sim \log{n}$.
However, there are much better arrangements for the problem on $\mathbb{S}^2$.

\begin{center}
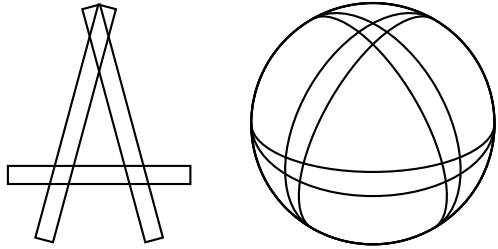
\begin{figure}[h!]
\begin{tikzpicture}[scale=0.8]
\draw[thick] (-6, -1) rectangle (-3, -0.7);
\draw[thick, rotate around={75:(-5,-1)}] (-6, -1) rectangle (-2, -0.7);
\draw[thick, rotate around={-75:(-4,-1)}] (-7, -1) rectangle (-3, -0.7);
\foreach \angle[count=\n from 1] in {0,120,240} {
    \begin{scope}[rotate=\angle]
\draw[thick] (0,0) circle (2) ;
    \path[thick, draw] (-2,0) arc [start angle=180,
                        end angle=360,
                        x radius=2cm,
                        y radius=0.8cm] ;
                            \path[thick,draw] (-2,0) arc [start angle=180,
                        end angle=360,
                        x radius=2cm,
                        y radius=1.2cm] ;
    \end{scope}
    };
\end{tikzpicture}
\caption{The difference between $\mathbb{R}^2$ and $\mathbb{S}^2$ illustrated: the curvature of $\mathbb{S}^2$ increases transversality, which decreases the area of intersection.}
\end{figure}
\end{center}
\vspace{-20pt}

\begin{thm} Let $s \geq 0$. There exists $c_s > 0$ such that for every set of $n$ great circles 
$$ \sum_{i, j = 1 \atop i \neq j}^{n}{|C_i \cap C_j|^s} \geq c_s \begin{cases} n^{2 - 2s} \qquad &\mbox{if}~0 \leq s < 2 \\ n^{-2} \log{n} \qquad &\mbox{if}~s = 2\\ n^{1- 3s/2} \qquad &\mbox{if}~s > 2. \end{cases}$$
\end{thm}
All these results are sharp and we can obtain arrangements achieving these bounds as follows: pick $n$ points $p_1 , p_2, \dots, p_n \in \mathbb{S}^2$ in such a way that 
$$ \min_{1 \leq i, j \leq 1 \atop i \neq j}{  \| \pm p_i \pm p_j \|} \gtrsim \frac{1}{\sqrt{n}}$$
and then define $n$ great circles by defining $p_i$ and $-p_i$ to be the poles and take the 'equator'. 
One possible formulations of the Kakeya 
problem in $\mathbb{R}^2$ is as follows: given $n$ rectangles of size $1 \times 1/n$ such that any two intersect at an angle $\gtrsim n^{-1}$, how small can their union be? In particular, it is possible arrange rectangles in such a way that any two meet at an angle of at least $\gtrsim n^{-1}$ and the total area they cover is much smaller than the sum of the areas of the rectangles. This is not possible for the
analogous problem on $\mathbb{S}^2$.

\begin{corollary} Given $n$ great circles on $\mathbb{S}^2$ such that any two intersect at an angle $\geq 1/(100\sqrt{n})$, the union of their $1/n-$neighborhoods satisfies
$$ \left| \bigcup_{i=1}^{n}{ C_i} \right| \gtrsim 1.$$
\end{corollary}

\textit{Notation and Outline.} As for notation, we will not be concerned with the exact size of constants and will use $\lesssim, \gtrsim$ and $\sim$ throughout the paper. Here, $A \lesssim B$ means that $A \leq c B$
for some universal constant $0 < c < \infty$ that does not depend on any of the other parameters. $A \sim B$ means $A \lesssim B$ and $B \lesssim A$. Section \S 2 discusses the main idea behind
the proof, how it relates to Riesz energies and various other related problems. Section \S 3 gives a proof of the Proposition, Section \S 4 the proof of the main statement and Section \S 5 discusses sharpness of the result and gives a proof of the Corollary.

\section{Sketch of the proof -- related and open problems} The proofs hinge on the following observation: if  $C_1, C_2$ are the $1/n-$neighborhood of two great circles, then we may associate
to them their respective 'north poles' $p_1, p_2$ (see Fig. 3). As long as their angle of intersection satisfies $\alpha \gtrsim 1/n$, a rough approximation of the volume of intersection is $ |C_1 \cap C_2| \sim n^{-2} \alpha^{-1}$. At the same time, the angle of intersection also determines the geodesic distance between
the north poles $p_1, p_2$ via $ d_{\mathbb{S}^2}(p_1, p_2) = \alpha$. The geodesic distance on $\mathbb{S}^2$ and the Euclidean distance in $\mathbb{R}^3$ are equivalent up to constants
$$ \frac{2}{\pi} \|p_1 - p_2\| \leq d_{\mathbb{S}^2}(p_1, p_2) \leq \|p_1 - p_2\|.$$

\begin{center}
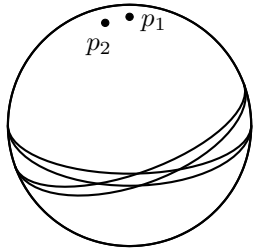
\begin{figure}[h!]
\begin{tikzpicture}[scale=0.8]
\foreach \angle[count=\n from 1] in {0,20} {
    \begin{scope}[rotate=\angle]
\draw[thick] (0,0) circle (2) ;
    \path[thick, draw] (-2,0) arc [start angle=180,
                        end angle=360,
                        x radius=2cm,
                        y radius=0.8cm] ;
                            \path[thick,draw] (-2,0) arc [start angle=180,
                        end angle=360,
                        x radius=2cm,
                        y radius=1cm] ;
    \end{scope}
    };
\filldraw (0,1.8) circle (0.06cm);
\filldraw (-0.4,1.7) circle (0.06cm);
\node at (0.4, 1.7) {$p_1$};
\node at (-0.5, 1.3) {$p_2$};
\end{tikzpicture}
\caption{Controlling $|C_1 \cap C_2|$ via $\|p_1 - p_2\|^{-1}$.}
\end{figure}
\end{center}
This suggests the approximation
$$ |C_1 \cap C_2| \sim \frac{1}{n^2} \frac{1}{\alpha} = \frac{1}{n^2} \frac{1}{d_{\mathbb{S}^2}(p_1, p_2)} \sim  \frac{1}{n^2} \frac{1}{\| p_1 -  p_2 \|}$$
and motivates the idea that
$$ \min_{C_1, \dots, C_n} \sum_{i, j = 1 \atop i \neq j}^{n}{|C_i \cap C_j|^s} \qquad \mbox{could be very similar to} \qquad  \min_{p_1, \dots, p_n} \frac{1}{n^{2s}}  \sum_{i, j = 1 \atop i \neq j}^{n}{ \frac{1}{\|p_i - p_j\|^s} }.$$
This is a classical problem ('point sets minimizing the $s-$Riesz energy') that has been very widely studied.
It was first considered by J. J. Thomson \cite{thomson} in 1904 for $s=1$ as a model for the behavior of $n$ charged
electrons on the sphere. It has only been solved for small values of $n$ and is Problem 7 on Smales's list of 'Mathematical Problems for the
Next Century' \cite{smale}. A concise survey of existing results has become impossible,
there are now hundreds of papers in mathematics, chemistry and physics on that subject  (see e.g. the  surveys of Hardin \& Saff \cite{hardin}, Saff \& Kuijlaars \cite{saff} and Brauchart \& Grabner \cite{brauch})
and there are a surprising number of applications in mathematics alone (quadrature rules, starting points for Newton’s method, finite element tessellations,...).\\

\textbf{Open problems: Riesz energies.} We note that the two problems are very similar but far from identical: in the problem of minimizing Riesz energies, the quantity $\|p_i - p_j\|^{-s}$ becomes unbounded if $p_i$ moves closer to
$p_j$. No such singularity is present in the geometric problem since $|C_i \cap C_j| \lesssim |C_i| \sim 1/n$. Moreover, in replacing great circles by poles, the interaction is between pairs
of antipodal points instead of pairs of points. \\

\begin{quote} \textit{Open problem.} Is it possible to use techniques from the study of point sets with minimal $s-$Riesz energy on $\mathbb{S}^2$ to get information about minimizers of the minimum overlap problem?\\
\end{quote}

While the Riesz energy $\|p_i - p_j\|^{-s}$ is natural from a physical point of view, it is perhaps less clear how well it interacts with the geometry of $\mathbb{S}^2$ (with the
possible exception of $s=-1$, see Stolarsky's invariance principle \cite{stolarsky}). Indeed, in the case of Coulomb energy $s=1$ it is even difficult to make definite
statements about relatively small values of $n$ (already $n=5$ is highly nontrivial and was only recently solved by Schwartz \cite{schw}). For large values of $n$ numerical computations
suggest a complicated asymptotic picture: the Voronoi cells of the arising configurations are mostly hexagonal but there are some 'scars' comprised of pentagonal
and heptagonal Voronoi cells (see Calef, Griffiths, Schulz, Fichtl \& Hardin \cite{many}). In contrast, we believe that the quantity $|C_i \cap C_j|$ may be more complicated to write
in closed form (see Lemma 2) but exceedingly natural from a geometric point of view which \textit{might} give rise to simpler optimal configurations.\\

\begin{quote} \textit{Open problem.} Are minimizers of the minimum overlap problem more structured than minimal energy configurations of the Riesz energy?\\
\end{quote}

One reason why it is conceivable there might be some additional structure is that one can interpret the $1/n-$neighborhood as a smoothing parameter: let us now consider
the $\delta-$neighborhoods $C_{1, \delta}, C_{2, \delta}, \dots, C_{n, \delta}$ of $n$ fixed great circles where no two great circles coincide and let $p_1, \dots, p_n$ denote one of their 'poles' (again, it does not matter which
one of the two is chosen). Then
$$ \lim_{\delta \rightarrow 0}{ \frac{1}{\delta^2} \sum_{i, j = 1 \atop i \neq j}^{n}{|C_{i,\delta} \cap C_{j,\delta}|^s}} =  \sum_{i, j = 1 \atop i \neq j}^{n}{  
 \frac{1}{   \left(   1 - \left\langle p_i, p_j \right\rangle^2 \right)^{s/2}}   },$$
This is very easy to see and will follow as a byproduct of our main argument.
 It introduces a curious variant of Riesz energies insofar as in minimizing configurations points not only repel their neighbors
but also repel points close to their own antipodal point: the limit $s \rightarrow \infty$ can be understood as a circle packing problem in the Grassmannian
$G(3,1)$ (see Conway, Hardin \& Sloane \cite{chs} for some numerical results). In particular, the minimum overlap problem can
be understood as a relaxation of this packing problem. It might also be natural to replace $1/n-$neighborhoods
by smoother cut-off functions and use the interpretation present in the next section to replace the geometric problem by an estimate
on $L^2-$norms.\\

\textbf{Open problems: $L^p-$norms.} The sum of total mutual overlap ($s=1$) is essentially equivalent to the $L^2-$norm of the sum of characteristic functions: using $\chi_{C_i}$ to denote the characteristic function of $C_i$ on $\mathbb{S}^2$, we see that
\begin{align*}
1 +  \sum_{i, j = 1 \atop i \neq j}^{n}{|C_i \cap C_j|}  &=    \int_{\mathbb{S}^2}{ \sum_{i= 1}^{n}{\chi_{C_i}^2} +
\sum_{i,j= 1 \atop i \neq j}^{n}{\chi_{C_i} \chi_{C_j}} dx} = \int_{\mathbb{S}^2}{ \left( \sum_{i= 1}^{n}{\chi_{C_i}} \right)^2 dx}.
\end{align*}
This means there are arrangements of great circles where
$$  \left\|   \sum_{i= 1}^{n}{\chi_{C_i}} \right\|_{L^1(\mathbb{S}^2)} \sim 1 \sim \left\|   \sum_{i= 1}^{n}{\chi_{C_i}} \right\|_{L^2(\mathbb{S}^2)}.$$
 It would be very interesting to understand the behavior of the $L^p-$norm for some
$p > 2$ and this seems to be a very difficult problem. The special case $L^{\infty}$ is related to an old combinatorial problem in discrete geometry
posed by Motzkin and (independently) W. M. Schmidt (see \cite{beck}).
\begin{center}
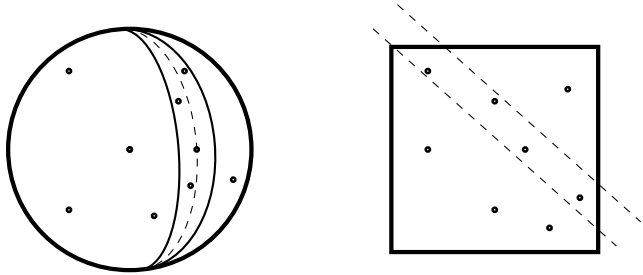
\begin{figure}[h!]
\begin{tikzpicture}[scale=0.8]
\draw[ultra thick] (0,0) circle (2) ;
\foreach \angle[count=\n from 1] in {95} {
    \begin{scope}[rotate=\angle]
    \path[thick, draw] (-2,0) arc [start angle=180,
                        end angle=360,
                        x radius=2cm,
                        y radius=0.8cm] ;
                            \path[dashed,draw] (-2,0) arc [start angle=180,
                        end angle=360,
                        x radius=2cm,
                        y radius=1.1cm] ;
                            \path[thick,draw] (-2,0) arc [start angle=180,
                        end angle=360,
                        x radius=2cm,
                        y radius=1.4cm] ;
    \end{scope}
    };
\draw[ultra thick] (0,0) circle (0.02) ;
\draw[ultra thick] (0.4,-1.1) circle (0.02) ;
\draw[ultra thick] (1.7,-0.5) circle (0.02) ;
\draw[ultra thick] (0.8,0.8) circle (0.02) ;
\draw[ultra thick] (0,0) circle (0.02) ;
\draw[ultra thick] (-1,-1) circle (0.02) ;
\draw[ultra thick] (-1,1.3) circle (0.02) ;
\draw[ultra thick] (1,-0.6) circle (0.02) ;
\draw[ultra thick] (1.1,0) circle (0.02) ;
\draw[ultra thick] (0.9,1.3) circle (0.02) ;

\draw [ultra thick] (4.3,-1.7) rectangle (7.7,1.7);
\draw[ultra thick] (4.9,1.3) circle (0.02) ;
\draw[ultra thick] (6.9,-1.3) circle (0.02) ;
\draw[ultra thick] (7.2,1) circle (0.02) ;
\draw[ultra thick] (4.9,0) circle (0.02) ;
\draw[ultra thick] (6,0.8) circle (0.02) ;
\draw[ultra thick] (6,-1) circle (0.02) ;
\draw[ultra thick] (7.4,-0.8) circle (0.02) ;
\draw[ultra thick] (6.5,0) circle (0.02) ;
\draw [dashed] (4, 2) -- (8, -1.6);
\draw [dashed] (4+0.4, 2+0.4) -- (8+0.4, -1.6+0.4);
\end{tikzpicture}
\caption{Finding strips that contains many points.}
\end{figure}
\end{center}
\vspace{-10pt}
 Associating to every great circle a 'north pole', it is not difficult to see that three
great circles meet in a point if and only if their north poles are contained in a another great circle. The $L^{\infty}-$question
is therefore largely equivalent to the following question: given a set of $n$ points on $\mathbb{S}^2$ and all 'zones' ($1/n-$neighborhoods
of a a great circle), which zone contains the most points?  The problem is easier to state in $\mathbb{R}^2$: given $n$ points in $[0,1]^2$, how many points are in the $2 \times (1/n)-$rectangle containing the most points? The dual question is as follows: given a set of $n$ points how thick is is the thinnest strip containing 3 points (where a strip is simply a neighborhood of an infinite line)? This problem is (independently) due to Motzkin and W. Schmidt and it is not even clear whether the thickness is $o(1/n)$ (the bound $\lesssim 1/n$ follows immediately from pigeonholing). A conditional result is due to Beck \cite{beck}. Another
family of problems that are vaguely related can be found in results surrounding Tarski's plank problem and its various generalizations (see e.g. Bezdek \cite{bezdek}).
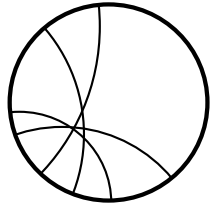
\begin{figure}[h!]
\begin{center}
\begin{tikzpicture}[scale=1.3]
  \tkzDefPoint(0,0){O}
  \tkzDefPoint(1,0){A}
  \tkzDrawCircle[ultra thick](O,A)
  \tkzClipCircle(O,A)
  \tkzDefPoint(-1,-0.1){z1}
  \tkzDefPoint(0,-0.8){z2}
  \tkzDrawCircle[ thick,orthogonal through=z1 and z2](O,A)
  \tkzDefPoint(-0.1,1){z5}
  \tkzDefPoint(-1,-1){z6}
  \tkzDrawCircle[ thick,orthogonal through=z5 and z6](O,A)
  \tkzDefPoint(-0.9,1){z7}
  \tkzDefPoint(-2,-2){z8}
  \tkzDrawCircle[ thick,orthogonal through=z7 and z8](O,A)
  \tkzDefPoint(-0.85,-0.3){z9}
  \tkzDefPoint(0.7,-0.85){z10}
  \tkzDrawCircle[ thick,orthogonal through=z9 and z10](O,A)
\end{tikzpicture}
\end{center}
\caption{Four geodesics in the hyperbolic plane any two of which intersect.}
\end{figure}

One interpretation of the main result is that it quantifies the extent to which positive curvature may increase the transversality of intersecting geodesics: can anything be said about
the problem in the hyperbolic plane?

\section{Proof of the Proposition}
\subsection{Lemma} We start by proving an elementary inequality.
\begin{lemma} Let $0 \leq x_1 < x_2 < \dots < x_n < 2\pi$. Then we have the sharp inequality
$$ \sum_{i,j = 1 \atop i \neq j}^{n}{\frac{1}{|x_i - x_j|_{}} + \frac{1}{2\pi-|x_i-x_j|_{}}} \geq  \frac{n^2}{\pi} \sum_{k=1}^{n-1}{\frac{1}{k}}.$$
\end{lemma}
\begin{proof} Instead of working with the $x_i$, we will instead consider $h_i = x_{i+1} - x_i$ for $1 \leq i \leq n-1$ and
$h_n = x_1 + (2\pi - x_n)$ and will set $h_{n+k} := h_{k}$ for $1 \leq k \leq n$. Then it is not difficult to see that
$$
  \sum_{i,j = 1 \atop i \neq j}^{n}{\frac{1}{|x_i - x_j|_{}} + \frac{1}{2\pi-|x_i-x_j|_{}}} = 2\sum_{i,j = 1 \atop i < j}^{n}{\frac{1}{|x_i - x_j|_{}} + \frac{1}{2\pi-|x_i-x_j|_{}}}
= 2\sum_{k=1}^{n-1}{ \sum_{m=1}^{n}{ \left(\sum_{i = m}^{i = m + k-1}{h _i}  \right)^{-1}}}.$$
We now observe that for fixed $k$
$$  \sum_{m=1}^{n}{ \left(\sum_{i = m}^{i = m + k-1}{h _i}  \right)} = k \sum_{m=1}^{n}{h _m} = 2 \pi k.$$
The inequality of arithmetic and harmonic mean for $n$ positive real numbers can be written as
$$ \frac{1}{a_1} + \frac{1}{a_2} + \dots + \frac{1}{a_n} \geq \frac{n^2}{a_1 + \dots + a_n}.$$
This gives that
$$ \sum_{m=1}^{n}{ \left(\sum_{i = m}^{i = m + k-1}{h _i}  \right)^{-1}} \geq \frac{n^2}{2 \pi k}$$
from which we can deduce that 
$$ 2\sum_{k=1}^{n-1}{ \sum_{m=1}^{n}{ \left(\sum_{i = m}^{i = m + k}{h _i}  \right)^{-1}}} \geq  \sum_{k=1}^{n-1}{  \frac{n^2}{\pi k}} = \frac{n^2}{\pi} \sum_{k=1}^{n-1}{\frac{1}{k}}.$$
Sharpness follows immediately for $h_1 = h_2 = \dots = h_n = (2\pi)/n.$
\end{proof}

\subsection{Proof of the Proposition}
\begin{proof} Given $1/n-$neighborhoods $T_1, \dots, T_n$ of $n$ lines in $\mathbb{R}^2$ we can assume w.l.o.g.
that $T_1$ is given by the $x-$axis. If any two lines coincide, then the desired quantity is unbounded as their overlap has infinite area. We may thus assume that any two distinct lines intersect in a point. We will now study all other lines by analyzing their angles of intersection with $T_1$. This gives a list $0 = \alpha_1 < \alpha_2, \alpha_3, \dots, \alpha_n < 2\pi.$ The natural notion of distance is the toroidal distance
$$ | \alpha_i - \alpha_j |_{\mathbb{T}} = \min\left(|\alpha_i - \alpha_j| , 2\pi -|\alpha_i - \alpha_j|\right).$$
Clearly, $T_i$ and $T_j$ will then intersect and the smaller angle of intersection (i.e. the one satisfying $0 \leq \alpha \leq \pi/2$) is given by $\min(|\alpha_i - \alpha_j|_{\mathbb{T}}, \pi - |\alpha_i - \alpha_j|_{\mathbb{T}})$. 
 We first compute the area of intersection of the $1/n-$neighborhoods of two lines
meeting at an angle $0 \leq \alpha \leq \pi/2$ using elementary trigonometry (see Fig. 6): the line $\ell_1$ has length $2/n$, which implies that $\ell_2$ has length $2/(n \sin{\alpha})$ while the height of the rhomboid
is given by $2/n$ and 
$$ |T_i \cap T_j| = \frac{2}{n \sin{\alpha}} \frac{2}{n} =  \frac{4}{n^2} \frac{1}{\sin{\alpha}} \geq \frac{ 4/n^2 }{\alpha}.$$

\begin{figure}[h!]
\begin{center}
\begin{tikzpicture}[scale=1.3]
\draw [ultra thick] (0,0) -- (4,0);
\draw [ultra thick] (0,1) -- (4,1);
\draw [ultra thick] (0,-0.3) -- (2,1.2);
\draw [ultra thick] (2,-0.3) -- (4,1.2);
   \draw [black,thick,domain=0:25] plot ({1.2*cos(\x)}, {1.2*sin(\x)});
\node at (0.9, 0.15) {{\Large $\alpha$}};
\draw (1.72,1) -- (2.55,+.1);
   \draw [black,thick,domain=180:203] plot ({4.2+1.2*cos(\x)}, {1+1.2*sin(\x)});
\node at (1.9, 0.5) {{\Large $\ell_1$}};
\node at (2.6, 1.3) {{\Large $\ell_2$}};
\node at (3.2, 0.8) {{\Large $\alpha$}};
   \draw [black,thick,domain=38:130] plot ({2.5+0.4*cos(\x)}, {0.1+0.4*sin(\x)});
\filldraw (2.53, 0.3) circle (0.02cm);
\end{tikzpicture}
\end{center}
\caption{The intersection $T_i \cap T_j$.}
\end{figure}
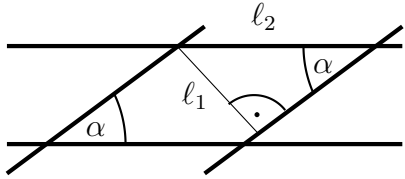

The lower bound for $0 < s < 1$ is trivial, since $|T_i \cap T_j|^s \gtrsim n^{-2s}$ is always true.  We start with $s=1$ and write
$$ \sum_{i, j = 1 \atop i \neq j}^{n}{|T_i \cap T_j|} \geq \frac{4}{n^2} \sum_{i, j = 1 \atop i \neq j}^{n}{\frac{1}{\min(|\alpha_i - \alpha_j|_{\mathbb{T}}, \pi - |\alpha_i - \alpha_j|_{\mathbb{T}})}} \geq \frac{4}{ n^2} \sum_{i, j = 1 \atop i \neq j}^{n}{\frac{1}{|\alpha_i - \alpha_j|_{\mathbb{T}}}   } 
$$
and thus, using $2\min(a,b) \geq 1/a + 1/b$ and Lemma 1
\begin{align*}
 \frac{4}{ n^2} \sum_{i, j = 1 \atop i \neq j}^{n}{\frac{1}{|\alpha_i - \alpha_j|_{\mathbb{T}}}   } &=  \frac{4}{ n^2} \sum_{i, j = 1 \atop i \neq j}^{n}{\frac{1}{\min(|\alpha_i - \alpha_j|, 2\pi - |\alpha_i - \alpha_j|)  }   }    \\
&\geq  \frac{2}{n^2} \sum_{i, j = 1 \atop i \neq j}^{n}{\frac{1}{|\alpha_i - \alpha_j|_{}}   +    \frac{1}{2\pi - |\alpha_i - \alpha_j|_{}}   } \\
&\geq \frac{2}{\pi}\sum_{k=1}^{n-1}{\frac{1}{k}} \sim \log{k}.
\end{align*}
It remains to consider $s > 1$. We simplify the computation by replacing the sum between all pairs by the sum of all adjacent pairs. Clearly, using $\alpha_{n+1} := \alpha_1$, 
$$   \sum_{i, j = 1 \atop i \neq j}^{n}{|T_i \cap T_j|}  \gtrsim \frac{1}{n^{2s}}\sum_{i,j = 1 \atop i \neq j}^{n}{\frac{1}{|\alpha_i - \alpha_j|_{}^s}} 
\geq  \frac{1}{n^{2s}} \sum_{i= 1}^{n}{\frac{1}{|\alpha_{i+1} - \alpha_i|_{}^s}}.$$
H\"older's inequality gives that
$$ \sum_{i= 1}^{n}{\frac{1}{| \alpha_{i+1} -  \alpha_i|_{}}} \leq \left( \sum_{i= 1}^{n}{\frac{1}{| \alpha_{i+1} -  \alpha_i|_{}^s}} \right)^{\frac{1}{s}} n^{\frac{s-1}{s}}$$
and therefore
$$  \frac{1}{n^{2s}} \sum_{i= 1}^{n}{\frac{1}{| \alpha_{i+1} -  \alpha_i|_{}^s}} \geq  \frac{1}{n^{2s}} \frac{1}{n^{s-1}} \left(\sum_{i= 1}^{n}{\frac{1}{| \alpha_{i+1} -  \alpha_i|_{}}}\right)^s.$$
Another application of the arithmetic-harmonic inequality gives
$$ \sum_{i= 1}^{n}{\frac{1}{| \alpha_{i+1} -  \alpha_i|_{}}} \geq \frac{n^2}{2\pi}$$
from which the result follows.
\end{proof}

\section{Proof of the Theorem}

\subsection{Spherical trigonometry.}
We start by simplifying the precise closed form for $|C_i \cap C_j|$, where $C_i, C_j$ are $1/n-$neighborhoods
of two great circles meeting at an angle $\alpha$. 
\begin{lemma}[Fodor, V\'igh \& Zarn\'ocz, \cite{fodor}] If $n \geq 2$ and $2/n \leq \alpha \leq \pi/2$, then
\begin{align*}
|C_i \cap C_j| &= 4 \sin{\left(n^{-1}\right)} \arcsin{\left( \frac{1-\cos{(\alpha)}}{\cot{\left(n^{-1}\right)} \sin{(\alpha)}} \right)}  +  4 \sin{\left(n^{-1}\right)} \arcsin{\left( \frac{1+\cos{(\alpha)}}{\cot{\left(n^{-1}\right)} \sin{(\alpha)}} \right)}\\
&-  2 \arccos{ \left( \frac{\cos(\alpha) - \sin^2{(n^{-1})}}{\cos^2{(n^{-1})}}\right)} -  2 \arccos{ \left( \frac{-\cos(\alpha) - \sin^2{(n^{-1})}}{\cos^2{(n^{-1})}}\right)} + 2\pi
\end{align*}
\end{lemma}
The condition $\alpha \geq 2/n$ is necessary because the two neighborhoods otherwise never separate and the geometric structure is different.
We would ultimately like to estimate $|C_i \cap C_j| \sim 4/((\sin{\alpha}) n^2)$, which is the corresponding area on the flat
background. It is clear that for fixed $\alpha > 0$, this will indeed be the arising limiting behavior as $n \rightarrow \infty$. 
However, even for fixed $n$, we expect effects from curvature 
to be relatively harmless as $\alpha$ approaches $n^{-1}$: if $\alpha \sim n^{-1}$, we obtain $4/((\sin{\alpha}) n^2) \sim 4/n$ while the arising area $|C_i \cap C_j|$ is also of order $\sim 1/n$. This approximation clearly
breaks down as soon as $\alpha \ll 1/n$ in which case the trivial bound is sharp up to constants and
$$ |C_i \cap C_j| \sim |C_i| \sim \frac{1}{n}.$$
\begin{lemma} If $n \geq 2$ and $2/n \leq \alpha \leq \pi/2$, then
$$
\frac{4}{n^2}  \frac{1}{\sin{\alpha}}  \leq |C_i \cap C_j| \leq \frac{4(\pi -2) }{n^2}  \frac{1}{\sin{\alpha}}.
$$
\end{lemma}

\begin{proof}[Sketch of an argument] The lower bound is sharp as $\alpha$ tends to $\pi/2$ as it is simply the area of intersection on a flat background. On the curved background
the geodesics are closer together than they would be on the flat background, which leads to more interaction and suggests that $(\sin{\alpha(C_i, C_j)})|C_i \cap C_j|$ is monotonically decreasing in $\alpha$. This requires
us to understand $\alpha = 2/n$ 
$$ \mbox{and for}~\alpha = \frac{2}{n}~\mbox{we have} ~  \lim_{n \rightarrow \infty}{ |C_i \cap C_j| \frac{4n^2}{\sin{\alpha}} }= \pi -2,$$
with convergence from below. The details follows from lengthy but straightforward computations.
\end{proof}

\subsection{Proof of Theorem 2}
\begin{proof} The case $0 \leq s < 2$ is trivial since
$$ \sum_{i, j = 1 \atop i \neq j}^{n}{|C_i \cap C_j|^s} \geq \sum_{i, j = 1 \atop i \neq j}^{n}{\frac{1}{n^{2s}}} \sim n^{2-2s}.$$
We will now deal with $s > 2$. We first associate to every great circle a 'north pole' and call this collection of points
$p_1, p_2, \dots, p_n$. Lemma 3 implies, \textit{assuming} the angle of intersection
is $\alpha \gtrsim 1/n$,
$$ |C_i \cap C_j| \sim \frac{1}{n^2} \frac{1}{\sin{\alpha}} \sim \frac{1}{n^2}  \frac{1}{\alpha} \sim  \frac{1}{n^2}  \max \frac{1}{\| \pm p_i \pm p_j\|} \gtrsim \frac{1}{n^2} \frac{1}{\|  p_i - p_j\|}.$$
This is now the classical expression for Riesz energies; it remains to show that we can indeed assume that the
 the angle of intersection is $\alpha \gtrsim 1/n$. We will argue that the set of exceptions 
$$ A = \left\{  1 \leq i \leq n: \exists_{j \neq i} ~~\|p_j - p_i\| \leq \frac{1}{n}\right\}  \qquad \mbox{is small:} \quad \#A \lesssim n^{1-s/2}.$$
This follows immediately from
$$\sum_{i, j = 1 \atop i \neq j}^{n}{ |C_i \cap C_j|^{s}} = \sum_{i=1}^{n}{ \sum_{j=1 \atop j \neq i}^{n}{ |C_i \cap C_j|}} \geq \sum_{i \in A}^{}{ \sum_{j=1 \atop j \neq i}^{n}{ |C_i \cap C_j|^s}} 
\geq   \sum_{i \in A}^{}{n^{-s} }  =  |A| n^{-s}$$
because either we now have $|A| \lesssim n^{1-s/2}$ or we have already shown the desired result. However, if $|A| \lesssim n^{1-s/2},$ we can simply remove the points in $A$ from the set and obtain
a well-separated set of $\sim n$ points for which the desired approximation holds (which, by an abuse of notation, we again call $p_1, p_2, \dots, p_n$ since the subsequent argument is not sensitive
to the precise number of points).
The scaling dictates that the largest contribution comes from the nearest neighbor, so we can again estimate in the same way as we did before in the proof of the Proposition:
$$ \sum_{i, j = 1 \atop i \neq j}^{n}{  \frac{1}{\|p_i - p_j\|^s} }  \geq  \sum_{i = 1}^{n}{  \sup_{ j \neq i} {\frac{1}{\|p_i - p_j\|^s}  }} = 
  \sum_{i = 1}^{n}{  \frac{1}{ \inf_{ j \neq i} {  \|p_i - p_j\|^s}  }} .$$  
We use again H\"older's inequality to bound 
$$ \sum_{i = 1}^{n}{  \frac{1}{ \inf_{ j \neq i} {  \|p_i - p_j\|^2}}} \leq  \left( \sum_{i = 1}^{n}{  \frac{1}{ \inf_{ j \neq i} {  \|p_i - p_j\|^s}}} \right)^{\frac{2}{s}} n^{\frac{s-2}{s}}$$
from which we get that 
$$   \sum_{i = 1}^{n}{  \frac{1}{ \inf_{ j \neq i} {  \|p_i - p_j\|^s}  }}  \geq  \frac{1}{n^{\frac{s-2}{2}}}\left( \sum_{i = 1}^{n}{  \frac{1}{ \inf_{ j \neq i} {  \|p_i - p_j\|^2}}} \right)^{\frac{s}{2}}.$$
Finally, the inequality of arithmetic and harmonic mean gives
$$  \sum_{i = 1}^{n}{  \frac{1}{ \inf_{ j \neq i} {  \|p_i - p_j\|^2}}} \geq \frac{n^2}{ \sum_{i = 1}^{n}{   \inf_{ j \neq i} {  \|p_i - p_j\|^2}}} \gtrsim n^2,$$
where only the last inequality needs to be explained: it is easy to see that by placing a ball of size 
$$  \inf_{ j \neq i} { \frac{ \|p_i - p_j\|}{2}} \quad \mbox{around}~p_i~\mbox{we obtain a set of disjoint balls}$$
and the area of their union is therefore bounded from above
$$ \sum_{i = 1}^{n}{   \inf_{ j \neq i} {  \|p_i - p_j\|^2}} \lesssim 1.$$
Altogether
$$  \frac{1}{n^{2s}} \sum_{i, j = 1 \atop i \neq j}^{n}{  \frac{1}{\|p_i - p_j\|^s} }   \gtrsim  \frac{1}{n^{2s}}   \frac{1}{n^{\frac{s-2}{2}}}\left( \sum_{i = 1}^{n}{  \frac{1}{ \inf_{ j \neq i} {  \|p_i - p_j\|^2}}} \right)^{\frac{s}{2}}
\geq    \frac{1}{n^{2s}}    \frac{1}{n^{\frac{s-2}{2}}} n^s = n^{1 -\frac{3s}{2}}.$$
The remaining case $s=2$ works similarly: the simple argument
$$ |p_i - p_j| \leq \frac{1}{n} \implies |C_i \cap C_j|^2 \gtrsim \frac{1}{n^2}$$
implies that either
 $$ \# \left\{  1 \leq i \leq n: \exists_{j \neq i} ~~\|p_j - p_i\| \leq \frac{1}{n}\right\} \lesssim \log{n}$$
or we already have the desired result. This allows us to remove at most $\log{n}$ points and proceed with a well-separated set for which the desired asymptotic again holds. Then, however, we can invoke the corresponding result for Riesz $2-$energies (due to Kuijlaars \& Saff \cite{kui}) stating that
$$ \inf_{q_1, \dots, q_n \in \mathbb{S}^2}  \sum_{i, j = 1 \atop i \neq j}^{n}{  \frac{1}{\|q_i - q_j\|^2} } \gtrsim n^2 \log{n} \qquad \mbox{to deduce} \qquad 
 \frac{1}{n^4} \sum_{i, j = 1 \atop i \neq j}^{n}{  \frac{1}{\|p_i - p_j\|^2} } \gtrsim n^{-2} \log{n}.$$
\end{proof}

\textit{Remark.} The result of Kuijlaars \& Saff \cite{kui} could have also been invoked in the case $s \neq 2$ but we thought it advantageous to
give a self-contained argument since the proof (which is a variant of the argument given in \cite{kui} but not fundamentally novel) is nice and straightforward.

\section{Sharpness and Proof of the Corollary}

\subsection{Constructing examples}
We will place $n$ great circles implicitly by fixing the location of their north poles $p_1, p_2, \dots, p_n$ in such a way that
$$ \inf_{1 \leq i, j \leq n \atop i \neq j}{ \|  \pm p_i \pm  p_j \| } \gtrsim \frac{1}{\sqrt{n}}.$$
Since we are not interested in sharp constants, this can be done in many different ways -- one could simply take a greedy selection where
each selected point removes two spherical caps of area $\sim 1/n$ which allows, for suitable constants, to get $n$ pairs of points with the desired
property. Then
$$ \sum_{i, j = 1 \atop i \neq j}^{n}{|C_i \cap C_j|^s} \lesssim \frac{1}{n^{2s}}\sum_{i, j = 1 \atop i \neq j}^{n}{\frac{1}{\min  \|  \pm p_i \pm  p_j \|^s }},$$
where the minimum ranges over all 4 possible constellations of signs. We simplify by simply adding the points $-p_1, -p_2, \dots, -p_n$ to
the set via $p_{n+k} = -p_{k}$ for $1 \leq k \leq n$ and use 
$$  \sum_{i, j = 1 \atop i \neq j}^{n}{\frac{1}{\min  \|  \pm p_i \pm  p_j \|^s }} \leq  \sum_{i, j = 1 \atop i \neq j}^{2n}{\frac{1}{\|  p_i -  p_j \|^s }}.$$
We will now fix $p_i$ and find an upper bound on the contribution coming from interactions with that point. Clearly, the quantity is maximized
if all points are as close to $p_i$ as possible, however, all of them are also $\sim 1/\sqrt{n}-$separated. A simple geometric arguments shows that
$$ \# \left\{ 1 \leq j \leq n:  \frac{k}{\sqrt{n}} \leq  \| p_j - p_i\| \leq \frac{k+1}{\sqrt{n}} \right\} \lesssim k.$$\
To see this (see Fig. 6), we place a $\sim 1/\sqrt{n}-$ball around each point in the annulus. Since the points are $\sim 1/\sqrt{n}-$separated, these balls can be chosen so as not to overlap
from which we can conclude that 
$$  \# \left\{ 1 \leq j \leq n:  \frac{k}{\sqrt{n}} \leq  \| p_j - p_i\| \leq \frac{k+1}{\sqrt{n}} \right\}  \frac{1}{n} \lesssim \mbox{area of the annulus} \sim \frac{k}{\sqrt{n}} \frac{1}{\sqrt{n}}.$$

\begin{center}
\begin{figure}[h!]
\begin{tikzpicture}[scale = 0.6]
\draw[thick] (0,0) circle (3cm);
\draw[thick] (0,0) circle (2cm);
\filldraw (0,0) circle (0.08cm);
\node at (0.5, -0.3) {$p_i$};
\draw[fill=white,
         postaction = {pattern=north east lines,pattern color=gray}] (-1.5,1.9) circle (0.5);
\filldraw (-1.5,1.9) circle (0.05cm);
\draw[fill=white,
         postaction = {pattern=north east lines,pattern color=gray}] (-2.1,0.1) circle (0.5);
\filldraw (-2.1,0.1) circle (0.05cm);
\draw[fill=white,
         postaction = {pattern=north east lines,pattern color=gray}] (1.5,-1.9) circle (0.5);
\filldraw (1.5,-1.9) circle (0.05cm);
\draw[fill=white,
         postaction = {pattern=north east lines,pattern color=gray}] (1.9,1.5) circle (0.5);
\filldraw (1.9,1.5) circle (0.05cm);
\draw[fill=white,
         postaction = {pattern=north east lines,pattern color=gray}] (-0.5,-2.2) circle (0.5);
\filldraw (-0.5, -2.2) circle (0.05cm);
\draw[fill=white,
         postaction = {pattern=north east lines,pattern color=gray}] (2.8,0) circle (0.5);
\filldraw  (2.8,0) circle (0.05cm);
\draw[fill=white,
         postaction = {pattern=north east lines,pattern color=gray}] (0.2,2.7) circle (0.5);
\filldraw  (0.2,2.7) circle (0.05cm);
\draw[fill=white,
         postaction = {pattern=north east lines,pattern color=gray}] (-2.5,1.3) circle (0.5);
\filldraw  (-2.5,1.3) circle (0.05cm);
\draw[fill=white,
         postaction = {pattern=north east lines,pattern color=gray}] (-2,-2) circle (0.5);
\filldraw  (-2,-2) circle (0.05cm);
\draw[fill=white,
         postaction = {pattern=north east lines,pattern color=gray}] (2.3,-1) circle (0.5);
\filldraw  (2.3,-1) circle (0.05cm);
\draw[fill=white,
         postaction = {pattern=north east lines,pattern color=gray}] (-2,-0.95) circle (0.5);
\filldraw  (-2,-0.95) circle (0.05cm);
\end{tikzpicture}
\caption{Bounding the number of points at distance $k/\sqrt{n} \leq  \| p_j - p_i\| \leq (k+1)/\sqrt{n}$.}
\end{figure}
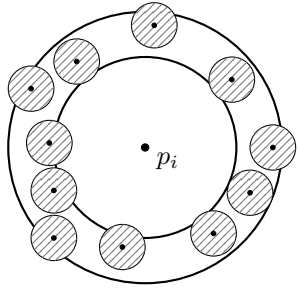
\end{center}

This allows us to compute
$$ \sum_{j = 1 \atop j \neq i}^{2n}{\frac{1}{\|  p_i -  p_j \|^s }} \lesssim \sum_{k=1}^{2 \sqrt{n} }{ \frac{k}{(\frac{k}{\sqrt{n}})^s}} = n^{s/2}  \sum_{k=1}^{2 \sqrt{n} }{  \frac{1}{k^{s-1}}} 
\sim n^{s/2} \begin{cases} n^{\frac{2-s}{2}} \qquad &\mbox{if}~0<s<2 \\
\log{n} \qquad &\mbox{if}~s=2\\
1 \qquad &\mbox{if}~s > 2.
\end{cases}
$$ 
Altogether, we have
$$ \frac{1}{n^{2s}}\sum_{i, j = 1 \atop i \neq j}^{n}{\frac{1}{\min  \|  \pm p_i \pm  p_j \|^s }} \lesssim \frac{n^{}}{n^{2s}}  \sum_{j = 1 \atop j \neq i}^{2n}{\frac{1}{\|  p_i -  p_j \|^s }} 
\lesssim \begin{cases} n^{2-2s} \qquad &\mbox{if}~0<s<2 \\
n^{-2} \log{n} \qquad &\mbox{if}~s=2\\
n^{1-\frac{3s}{2}} \qquad &\mbox{if}~s > 2.
\end{cases}$$

\subsection{Proof of the Corollary}

\begin{proof} We will again replace the geometric problem: given a set of $n$ great circles on $\mathbb{S}^2$ such that any two intersect at an angle $\gtrsim 1/\sqrt{n}$,
we can associate to each great circle a 'north pole
 (for this argument it is not important which of the two possible points is chosen but we are only allowed to pick one). As before, when constructing sharp
examples, we see that
$$ \sum_{i, j = 1 \atop i \neq j}^{n}{|C_i \cap C_j|} \lesssim \frac{1}{n^{2}}\sum_{i, j = 1 \atop i \neq j}^{n}{\frac{1}{\min  \|  \pm p_i \pm  p_j \| }} \lesssim 1.$$
Let $\chi_{C_i}$ to denote the characteristic function of $C_i$ and
$ C = \bigcup_{i=1}^{n}{C_i}.$
We see with Cauchy-Schwarz
\begin{align*}
 1 = n \cdot \frac{1}{n} \sim \sum_{i=1}^{n}{|C_i|} &= \int_{\mathbb{S}^2}{ \sum_{i= 1}^{n}{\chi_{C_i}} dx} =  \int_{C}{ \sum_{i= 1}^{n}{\chi_{C_i}} dx}  \sim   \left(\int_{C}{  \sum_{i= 1}^{n}{\chi_{C_i}} dx}\right)^2 \\
&\leq  |C|  \int_{C}{ \left( \sum_{i= 1}^{n}{\chi_{C_i}} \right)^2 dx} = |C| \left( 1 +  \sum_{i, j = 1 \atop i \neq j}^{n}{|C_i \cap C_j|}  \right) \lesssim |C|.
\end{align*}
\end{proof}

\subsection{The optimal configuration.}  We have not yet adressed the question of optimal configurations for fixed $n$. Given the intimate connection to Thomson's problem and its
obvious intrinsic geometric significance, the case $s=1$
seems most natural to study and our main result implies that
$$ c_1 \leq \inf_{n \in \mathbb{N}} \inf_{C_1, \dots, C_n} \sum_{i, j =1 \atop i \neq j}^{n}{|C_i \cap C_j|}   \leq \sup_{n \in \mathbb{N}} \inf_{C_1, \dots, C_n} \sum_{i, j =1 \atop i \neq j}^{n}{|C_i \cap C_j|} \leq c_2$$
for some universal constants $0 < c_1 \leq c_2 < \infty$. At this point, we have no nontrivial information about the problem, the constants involved or the structure of minimal configurations.
As already mentioned, understanding the fine structure of configurations minimizing $s-$Riesz energies has been a subject of great interest. In some cases (see, for example, Dahlberg \cite{dahlberg}, Kuijlaars \& Saff \cite{kui})
it is known that the points in the optimal configuration have to have a certain minimum distance between each other and in some cases (G\"otz \& Saff \cite{go}) it is known that they have to be asymptotically uniformly distributed.
The natural analogue for our problem is the question 
$$ \mbox{whether every optimal configuration has to satisfy} \quad \min_{i \neq j}{ \mbox{angle}(C_i, C_j)} \gtrsim \frac{1}{\sqrt{n}}?$$
This seems to be obviously true but nonetheless quite difficult to show: when minimizing the Riesz energy, the arguments incorporate $|x-y|^{-s}$ having a singularity to enforce some regularity whereas no such singularity is  present in the case of the overlap problem.

\subsection{The limit $\delta \rightarrow 0$.}
In this section we quickly derive
$$ \lim_{\delta \rightarrow 0}{ \frac{1}{\delta^2} \sum_{i, j = 1 \atop i \neq j}^{n}{|C_{i,\delta} \cap C_{j,\delta}|^s}} =  \sum_{i, j = 1 \atop i \neq j}^{n}{  
 \frac{1}{   \left(   1 - \left\langle p_i, p_j \right\rangle^2 \right)^{s/2}}   }.$$
\begin{proof}  It is easy to see that $|C_i \cap C_j|$ converges to the behavior on the flat background $\mathbb{R}^2$ as $\delta \rightarrow 0$ whenever
the angle between $C_i$ and $C_j$ is bigger than 0. This immediately implies 
that
$$  \lim_{\delta \rightarrow 0}{ \frac{1}{\delta^2} |C_{i,\delta} \cap C_{j,\delta}|^s} = \frac{1}{\sin{\alpha}},$$
where $\alpha$ is the angle between the two great circles. The angle between the two great circles is the same as the angle between the two poles (assuming they are contained in
the same hemisphere) and since $0 \leq \alpha \leq \pi/2$
$ \left\langle p_i, p_j \right\rangle = \|p_i\|\|p_j\| \cos{\alpha} = \cos{\alpha},$
we get that
$ 1 - \left\langle p_i, p_j \right\rangle^2 = \sin^2{\alpha}.$
This arising equation is automatically invariant under replacing $p_i$ by $-p_i$, which allows us to not distinguish between north poles and south poles.
\end{proof}

\end{document}